\documentclass[12pt]{amsart}

\usepackage{amssymb}

\usepackage{enumitem}

\usepackage{graphicx}

\makeatletter
\@namedef{subjclassname@2010}{%
  \textup{2010} Mathematics Subject Classification}
\makeatother

\usepackage[T1]{fontenc}

\numberwithin{equation}{section}

\newcommand{\R}{\mathbb{R}}

\newcommand{\D}{\mathbb{D}}
\newcommand{\Cm}{\mathbb{C}}
\newcommand{\intt}{\int\limits}
\newcommand{\summ}{\sum\limits}

\newcommand{\eps}{\varepsilon}

 \DeclareMathOperator{\dist}{dist}

\renewcommand{\phi}{\varphi}

\newtheorem{Thm}{Theorem}[section]
\newtheorem{theorem}[Thm]{Theorem}

\newtheorem{lemma}[Thm]{Lemma}

\newtheorem{remark}[Thm]{Remark}
\newtheorem{definition}{Definition}

\begin{document}


\baselineskip=17pt


\title[On the Heins Theorem]
{On the Heins Theorem}
\author{Aleksei Kulikov}
\address{Aleksei Kulikov,
\newline Department of Mathematics and Computer Science St.~Petersburg State University, St. Petersburg, Russia,
\newline {\tt lyosha.kulikov@mail.ru}
}
\thanks{ The work was supported by the Russian Science Foundation grant 17-11-01064.}
\date{}

\begin{abstract} { It is known that the famous {Heins Theorem (also known as the de Branges Lemma)} about the minimum of two entire functions of minimal type
does not extend to functions of finite exponential type. We study in detail pairs of entire functions $f, g$ of finite exponential type satisfying
$\sup_{z\in\mathbb{C}}\min\{|f(z)|,|g(z)|\}<\infty.$ It turns out that $f$ and $g$ have to be bounded on some {\it rotating half-planes}. We also obtain very close upper and lower bounds for possible rotation functions of these half-planes.
}
\end{abstract}

\maketitle

\section{Introduction { and main results}}

{ If two entire functions $f$ and $g$ of minimal exponential type are such that 
\begin{equation}
\sup_{z\in\mathbb{C}}\min\{|f(z)|, |g(z)|\}\leq 1,
\label{maineq}
\end{equation}
then at least one of them must be constant. This beautiful result, known to some experts as the de Branges Lemma, was proved by Maurice Heins in 1959 (see \cite[Theorem 5.1]{Heins}). It plays a crucial role in the proof of de Branges's Ordering Theorem \cite[Theorem 35]{de Branges} and has many applications in function theory and spectral theory
of differential operators (see e.g. \cite{Belov, Belov2}).  First of all we prove a small refinement of Heins's Theorem.
}
\begin{theorem}\label{minimal type}
Let $f, g$ be entire functions of minimal and finite exponential type, respectively,{ which satisfy \eqref{maineq}. Then either $f$ or $g$ is constant.}
\end{theorem}

This theorem obviously fails in the case of both functions having finite exponential type (consider $f(z) = e^z, g(z) = e^{-z}$). Nevertheless something can be said about domains where these functions are small. {  The main aim of this paper
is to study geometric properties of such domains.
\smallskip

A careful analysis of the proof of {Heins's Theorem} shows that the sets $\{|f(z)| \leq 1\}$ and $\{ |g(z)| \leq 1\}$
have to be close to half-planes.
}

\begin{theorem}\label{two arcs theorem}
{Let $f, g$ be nonconstant entire functions of finite exponential type { which satisfy \eqref{maineq}.} Then for any $\eps > 0$ there exists a set $E_\eps$ of finite measure 
 such that for $\tau \in \R_+ \backslash E_\eps$ the sets $\{|f(z)| \leq 1\}$ and $\{ |g(z)| \leq 1\}$ intersect circle of radius $e^\tau$ centered at the origin} in two opposite semicircles modulo a set of the circle of angular measure at most $\eps${, i.e. there exists a semicircle $C\subset \{|z| = e^{\tau}\}$ and a set $B$ of angular measure at most $\eps$ such that $$|f(z)| \le 1, |g(z)| > 1, z\in C\backslash B,$$ $$|f(z)| > 1, |g(z)| \le 1, z\in (-C)\backslash B.$$}
\end{theorem}


In view of this theorem it is reasonable to ask whether orientation of these semicircles can change{ when the radius tends to infinity}. We {are able to} answer this question in an affirmative way by constructing an entire function of finite exponential type which is bounded in some rotating half-plane, and we also give very close upper and lower bounds on possible rotation functions of this half-plane.

\begin{definition}\label{rotating half-planes}
{\normalfont Let $s:\R_{+} \to \R$ be a non-decreasing continuous function. We will say that $\Omega_s$ is a {\it rotating half-plane} with rotation function $s$ if}
\begin{equation*}
\Omega_s = \{ re^{i\alpha} \mid   r > 0,\, s(r) < \alpha < s(r) + \pi\}.
\end{equation*}
\end{definition}

\sloppy Note that the interior of the complement of a rotating half-plane $\Omega_s$ is a rotating half-plane with rotation function $s(r) + \pi$. Hence, if an entire function $f$ is bounded on a rotating half-plane $\Omega_s$, then $\sup_{z\in \Cm}\min\{|f(z)|, |f(-z)|\}<\infty$.

\medskip
{
{ Now we are ready to formulate the  main results of the paper.}

Put $h(r) = s(e^r)$. To avoid inessential technicalities we will always assume that $s\in C^2(\R_+)$.
\begin{theorem}\label{existence theorem}
For every rotating half-plane $\Omega_s$ such that $h'(x)\to 0$ as $x\to +\infty$ and 
\begin{equation}\label{sqint}
\intt_0^\infty h'(x)^2 dx < \infty
\end{equation}
 there exists a nonconstant entire function of finite exponential type bounded in $\Omega_s$.
\end{theorem}


It turns out that under additional regularity of the function $h$ condition \eqref{sqint} is also necessary for the existence of an entire function of finite exponential type bounded in $\Omega_s$.

\begin{theorem}\label{regular upper bound}
Let $\Omega_s$ be a rotating half-plane with rotation function $s$ and let $f$ be an entire function of finite exponential type bounded on $\Omega_s$. Assume that $s$ is such that

$$\intt_0^\infty h'(x)^2dx = \infty \text{\quad and\quad }  \intt_0^\infty |h''(x)| dx < \infty.$$

Then $f$ is constant.
\end{theorem}

In particular, if the function $h(x)$ is concave, then these two theorems give us that condition \eqref{sqint} is necessary and sufficient (note that $h'(x) \ge 0$ since we assume that $s(r)$ is increasing).

Proofs of  Theorems \ref{existence theorem}, \ref{regular upper bound} are based on the estimates for conformal mappings between infinite strips from \cite{War}.

Interestingly, we can deduce a version of  Theorem \ref{regular upper bound} without assuming any regularity of the function $s$ using estimates for harmonic measure in simply connected domains.
\begin{theorem}\label{upper boundintr}
Let $\Omega_s$ be a rotating half-plane with rotation function $s$ and let $f$ be an entire function of finite exponential type bounded on $\Omega_s$. If the function $s$ is such that $$\limsup\limits_{r\to\infty} \frac{s(r)}{\sqrt{\log r}} > 0,$$ then $f$ is constant.
\end{theorem}

{
}
}
\subsection*{Acknowledgments} The author would like to thank Yurii Belov and Alexander Borichev for many helpful discussions. {The author also would like to thank Mikhail Sodin for pointing out the paper of Heins \cite{Heins} where the aforementioned  result about two functions of minimal type was proved. {The author is grateful to the anonymous referees for numerous remarks which substantially improved the exposition, and  for the suggestion to look into the paper of Warschawski \cite{War} which allowed the extension and simplification of the proof of Theorem \ref{existence theorem} and the proof of  Theorem \ref{regular upper bound}.}
}

\section{Functions of minimal and finite exponential type}

In this section we prove Theorems \ref{minimal type} and \ref{two arcs theorem}.

\subsection{Functions of minimal exponential type} 

\sloppy {For the proof of Theorem \ref{minimal type} we need the following variant of Lemma 27 from \cite{Romanov} { (see also {\cite[Theorem 8.1]{Hayman}}).}
	
\begin{lemma}{}\label{main lemma}
Let $u: \Cm \to \R$ be a continuous subharmonic function satisfying the following conditions:

\begin{enumerate}
\begin{item}$ u(z) \ge 0$ and $u(0) > 0$,
\end{item}

\begin{item} $u$ is smooth in some neighbourhood of $0$,
\end{item}

\begin{item}
$u$ is not constant in any neighbourhood of $0$.
\end{item}

\end{enumerate}

Put 
$$V_u = \overline{\{ z\mid u(z) > 0\}},$$ 
$$m_u(R)=\sup\{|I|/(2\pi): I  \text{ is an open interval},\ Re^{ix}\in V_u, x\in I\},$$
$$\eta_u(s)=\frac{1}{m_u(e^s)}.$$


Then there exists $C = C(u)> 0$ such that for all $\tau > 0$
\begin{equation}\label{lemma estimate}
\intt_0^{2\pi} u^2\left(e^\tau e^{i\theta}\right)d\theta \ge C\intt_0^{\tau}\exp\left(\intt_0^{\tau'}\eta_u(s)ds\right)d\tau'.
\end{equation}
\end{lemma}
{ We tacitly assume that $m_u(R)=\infty$, $\eta_u(\log R)=0$ when $R\mathbb{T}\subset V_u$. }

\medskip

\begin{proof}[Proof of Theorem \ref{minimal type}]

Suppose that neither $f$ nor $g$ is constant.  
Put ${f_1(z) = f(z) - f(0)}$, $g_1(z) = g(z) - g(0)$. Since $f$ and $g$ are nonconstant, $f_1$ and $g_1$ are not identically zero and so for some $k, n \in \mathbb{N}$ and $C_1 > 0$ the functions $f_2(z) = C_1f_1(z)/z^k$ and $g_2(z) = C_1g_1(z)/z^n$ are entire and $|f_2(0)|, |g_2(0)| > 1$. 

Since $\min \{|f(z)|, |g(z)|\} \le 1$, for some $C_2 > 1$ we have $\min \{|f_2(z)|, |g_2(z)|\} \le \frac{C_2}{|z|}, |z| > 1$. So, if we consider $f_3(z) = f_2(C_2z)$ and $g_3(z) = g_2(C_2z)$ then they will satisfy $\min\{|f_3(z)|, |g_3(z)|\}\le \frac{1}{|z|} < 1,$ $|z| > 1$. On the other hand, $|f_3(0)|, |g_3(0)| > 1$. { It is easy to see from the maximum principle that neither $f_3$ nor $g_3$ is constant.}

\smallskip

Put $u(z) = \max\{0, \log |f_3(z)|\}$ and $v(z) = \max\{0, \log |g_3(z)|\}$. {Now we apply Lemma \ref{main lemma}
for $u$ and $v$.}

 We consider some function $q: \R_+ \to \R_+$ { which we choose later} and sum estimates \eqref{lemma estimate} for $u$ and $v$ with weights $q(\tau)$ and $1$. 
\begin{multline*}
\intt_0^{2\pi} q(\tau)u^2(e^{\tau}e^{i\theta}) + v^2(e^\tau e^{i\theta})d\theta \ge \\ C\intt_0^\tau q(\tau)\exp\left(\intt_0^{\tau'} \eta_u(s)ds\right) + \exp\left(\intt_0^{\tau'}\eta_v(s)ds\right)d\tau'.
\end{multline*}

{ Applying the  inequality between the arithmetic and geometric means to the right-hand side, we get}
\begin{multline}\label{long estimate}
\intt_0^{2\pi} q(\tau)u^2(e^{\tau}e^{i\theta}) + v^2(e^\tau e^{i\theta})d\theta \ge \\  2C\sqrt{q(\tau)}\intt_0^\tau \exp\left(\intt_0^{\tau'} \frac{\eta_u(s) + \eta_v(s)}{2}ds\right)d\tau'.
\end{multline}

{ We claim that  $\eta_u(s) + \eta_v(s) \ge 4$ for all $s > 0$. }

From the definition of $u$ and $v$ we have $|f_3(z)| \ge 1, z\in V_u$ and ${|g_3(z)| \ge 1, z\in V_v}$. Let $I, J\subset \R$ be intervals such that $e^se^{ix}\in V_u$, $x\in I$ and $ e^se^{ix}\in V_v$, $x\in J$.

Since $\min\{|f_3(e^se^{ix})|, |g_3(e^se^{ix}) |\} < 1$, $x\in\mathbb{R}$ either $|I| + |J| \le 2\pi$ or one of these intervals has length more than $2\pi$. But if, for example, $|I| > 2\pi$, then $|g_3(e^se^{ix})| < 1$, $x\in\mathbb{R}$. { This contradicts the maximum principle.} Hence, $|I| + |J| \le 2\pi$. Taking the supremum over $I$ and $J$ we get 
\begin{equation}
m_u(e^s) + m_v(e^s) \le 1.
\label{uveq}
\end{equation} 
{The claim now follows from the  inequality between the arithmetic and harmonic means.}

From $\eta_u(s) + \eta_v(s)\geq 4$ and \eqref{long estimate} we get
\begin{equation}\label{final estimate}
\intt_0^{2\pi} q(\tau)u^2(e^{\tau}e^{i\theta}) + v^2(e^\tau e^{i\theta})d\theta \ge C\sqrt{q(\tau)}(e^{2\tau} - 1).
\end{equation}

Now, since $f_3$ is of minimal exponential type and $g_3$ is of finite exponential type, we have $u(e^\tau e^{i\theta}) = o(e^\tau)$ and $v(e^{\tau}e^{i\theta}) = O(e^{\tau})$ and so we can choose $q$ such that $q(\tau)\to \infty$ as $\tau \to \infty$ and $q(\tau)u^2(e^{\tau}e^{i\theta}) = O(e^{2\tau})$. {But for this choice of $q$ the left-hand side of \eqref{final estimate} is $O(e^{2\tau})$ while right-hand side is not. We arrive at a contradiction.}
\end{proof}

\subsection{Functions of finite exponential type} { We fix two entire functions $f$ and $g$ of finite exponential type satisfying \eqref{maineq}.}


\begin{lemma}\label{eta estimate}

We have $\intt_0^\infty \left(\frac{\eta_u(s) + \eta_v(s)}{2} - 2\right)ds<\infty$ .
\end{lemma}

\begin{proof}
We consider estimate \eqref{long estimate} with $q(\tau) = 1$

\begin{equation}\label{simple long estimate}
\intt_0^{2\pi} u^2(e^\tau e^{i\theta}) + v^2(e^\tau e^{i\theta})d\theta \ge 2C \intt_0^\tau \exp\left(\intt_0^{\tau'}\frac{\eta_u(s)+\eta_v(s)}{2}ds\right)d\tau'.
\end{equation}

Suppose that $\intt_0^\infty \left(\frac{\eta_u(s) + \eta_v(s)}{2} - 2\right)ds = \infty$. Since  $\frac{\eta_u(s) + \eta_v(s)}{2} - 2\ge 0$ {(see the proof of Theorem \ref{minimal type})},  for any $C_3 > 0$ there exists $\tau_0 > 0$ such that
\begin{equation*}
\intt_0^\tau \left(\frac{\eta_u(s) + \eta_v(s)}{2}\right)ds > 2\tau + C_3, \quad \tau>\tau_0.
\end{equation*}
Thus, for $\tau > \tau_0$ we have
\begin{multline}\label{eq9}
\intt_0^{2\pi} u^2(e^\tau e^{i\theta}) + v^2(e^\tau e^{i\theta})d\theta \ge \\ 2C \intt_{\tau_0}^\tau \exp\left(2\tau' + C_3\right)d\tau' = Ce^{C_3}(e^{2\tau} - e^{2\tau_0}).
\end{multline}

{ The left hand side of \eqref{eq9} is bounded by $ce^{2\tau}$ for some $c=c(f,g) > 0$. We can choose $C_3$ such that $c < Ce^{C_3}$. We arrive at  a contradiction for big enough $\tau$.}

\end{proof}

\begin{lemma}\label{technical lemma}
If $x, y > 0$, $x + y\le 1$ and $|x - \frac{1}{2}| > \eps$ then $\frac{1}{x} + \frac{1}{y} > 4 + \delta$ for some $\delta = \delta(\eps) > 0$.
\end{lemma}

\begin{proof}
If $1 - x - y > \eps$, then $\frac{1}{x} + \frac{1}{y} \ge \frac{4}{x+y} > \frac{4}{1 - \eps}$. Otherwise
\begin{equation*}
\frac{1}{x} + \frac{1}{y} =  \frac{4}{(x + y) - \frac{(x +y - 2x)^2}{x + y}} > \frac{4}{1 - \eps^2},
\end{equation*}

Hence, if $\delta = \min\{\frac{4}{1 - \eps^2}-4, \frac{4}{1-\eps}-4\}$ we get the desired estimate.
\end{proof}


\begin{lemma}\label{long arc theorem}
For any $\eps > 0$ there exists $E_\varepsilon\subset\mathbb{R}_+$, $|E_\varepsilon|<\infty$ such that for any $\tau\in\mathbb{R}_+\setminus E_\varepsilon$ the set $\{ z: |z| = e^\tau,  |f_3(z)| \ge 1\}$ contains an arc of angular measure at least $\pi - \eps$.
\end{lemma}

\begin{proof}
We know that $m_u(e^\tau) + m_v(e^\tau) \le 1$ (see \eqref{uveq}). From Lemmas \ref{eta estimate} and \ref{technical lemma} it follows that for all $\tau > 0$ except for the set of finite measure $m_u(e^\tau) > \frac{1}{2} - \frac{\eps}{2\pi}$. By definition, this means that $V_u\bigcap \{z: |z| = e^{\tau}\}$ contains an arc of angular measure at least $\pi - \eps$. 
\end{proof}

\begin{lemma}
For any $\eps > 0$ there exists $E_\varepsilon\subset\mathbb{R}_+$, $|E_\varepsilon|<\infty$ such that for any $\tau\in\mathbb{R}_+\setminus E_\varepsilon$ the set $\{ |z| = e^\tau,  |f(z)| > 1\}$ contains an arc of angular measure at least $\pi - \eps$.
\label{l5}
\end{lemma}

\begin{proof}
We consider only those $\tau > \log C_2$ for which the statement of Lemma \ref{long arc theorem} is fulfilled  with $\tau' = \tau - \log C_2$. If  $|f(e^{\tau}e^{i\theta})| \le 1$  then $|f_3(e^{\tau'}e^{i\theta})| < 1$ 
and so $\{ \theta : |f_3(e^{\tau'}e^{i\theta})| \ge 1\}\subset\{ \theta : |f(e^{\tau}e^{i\theta})| > 1\}$. From Lemma \ref{long arc theorem} we get the result.
\end{proof}

{ Now we are ready to prove Theorem \ref{two arcs theorem}. We will prove a slightly more general result.}

\begin{theorem}
For any $\eps > 0$ there exists $E_\varepsilon\subset\mathbb{R}_+$, $|E_\varepsilon|<\infty$, such that  for any $\tau \in \mathbb{R}_+ \backslash E_\varepsilon$ there are two disjoint arcs $I$, $J$ of angular measure at least $\pi - \eps$ on the circle $\{ |z| = e^{\tau}\}$ such that $|f(z)| > 1$ and $|g(z)| \le 1$ on $I$ and $|g(z)| > 1$ and $|f(z)| \le 1$ on $J$.
\end{theorem}

\begin{proof}
{
In view of Lemma \ref{l5} we consider only those $\tau > 0$ for which there exist  arcs $I$, $J$ of angular measure at least $\pi - \eps$ such that for $|f(z)| > 1$, $z\in I$ and for $|g(z)| > 1$, $z\in J$.  From $\min \{|f(z)|, |g(z)|\} \le 1$ 
we get that $I\cap J=\varnothing$ and moreover $|f(z)| \le 1, z\in J$ and $|g(z)| \le 1, z\in I$. 
}
\end{proof}

{\section{Rotating half-planes}

{ In this section we prove Theorems \ref{existence theorem},  \ref{regular upper bound} and \ref{upper boundintr}.}
\subsection{Proofs of Theorems \ref{existence theorem} and \ref{regular upper bound}}

For continuous functions ${\phi_-(u) < \phi_+(u)}$ we will call the set

$$S = \{ z = u + iv\mid \phi_-(u) < v < \phi_+(u), u > M\}, -\infty \le M < +\infty$$
a semi-infinite strip bounded by the curves $v = \phi_+(u), v = \phi_-(u), u = M$. Let $\theta (u) = \phi_+(u) - \phi_-(u)$ and $\psi(u) = \frac{1}{2}(\phi_+(u) + \phi_-(u))$.

We will denote by $Z(w) = X(w) + iY(w)$ a conformal mapping from $S$ to $S_0 = \{ x + iy\mid |y| < \frac{\pi}{2}\}$ such that $\Re w$ and $X(w)$ tends to $+\infty$ simultaneously. The function $W(z) = U(z) + iV(z)$ will denote the inverse of $Z(w)$.





Now we are ready to formulate results from \cite{War} that we will use in the proofs of Theorems \ref{existence theorem} and \ref{regular upper bound}.
\begin{theorem}{\cite[Theorem III a]{War}}\label{Alf bound}
Let $w_1 = u_1 + iv_1, w_2 = u_2 + iv_2$ and $u_1 \le u_2$. Then 

\begin{equation}
\pi \intt_{u_1}^{u_2}\frac{du}{\theta(u)} \le x_2 - x_1 + 4\pi,
\end{equation}
where $x_k = X(w_k)$.
\end{theorem}
\begin{theorem}{\cite[Theorem IV a]{War}}
If $|\phi_{\pm}'(u)|\le m$, then for some $x_0$, for $x_0 \le x_1 \le x_2$ and $|y_{1, 2}| < \frac{\pi}{2}$ we have the bound
\begin{equation}\label{conf upper}
x_2 - x_1 \le \pi \intt_{u_1}^{u_2}  \frac{1 + \psi'^2(u)}{\theta(u)}du + \frac{\pi}{12}\intt_{u_1}^{u_2} \frac{\theta'^2(u)}{\theta(u)}du + 8\pi \left(1 + \frac{4}{3}m^2\right),
\end{equation}
where $u_{k} = U(x_k + iy_k)$.
\end{theorem}

\begin{theorem}{\cite[Theorem VI]{War}}
If $|\phi_{\pm}'(u)| \to 0, u\to +\infty$ and ${|\phi_{\pm}''(u)|\in L^1(\R)}$, then for $w_1 = u_1 + iv_1$ and $w_2 = u_2 + iv_2$, $u_1 < u_2$, we have
\begin{equation}\label{conf lower}
x_2 - x_1 \ge \pi \intt_{u_1}^{u_2}  \frac{1 + \psi'^2(u)}{\theta(u)}du - \frac{\pi}{4}\intt_{u_1}^{u_2} \frac{\theta'^2(u)}{\theta(u)}du + o(1),
\end{equation}
where $x_k = X(w_k)$ and $o(1)$ is uniform as $u_1\to +\infty$.
\end{theorem}

\begin{theorem}{\cite[Theorem X iii]{War}}\label{arg asymp}
If $|\phi_{\pm}'(u)| \to 0, u\to +\infty$, then for every $y, |y| < \frac{\pi}{2}$, the line $\Lambda_y = \{ z\mid \Im z = y\}$ is  mapped by $Z$ onto a curve $L_y$ which for big enough $u$ is represented by the equation
\begin{equation}
v = f_y(u) = \psi(u) + \frac{\theta(u)}{\pi}y + o(\theta(u)), u\to +\infty.
\end{equation}
\end{theorem}

\begin{proof}[Proof of Theorem \ref{existence theorem}]
We will construct a function $f$ in two steps. First of all, we will construct an unbounded analytic function $g:\bar{\Omega}\to \Cm$ of finite exponential type which is bounded on $\partial \Omega$ for some suitable domain $\Omega$. Then we will construct function $f$ as a Cauchy integral of $g$ over the boundary of $\Omega$ and show that $f$ is bounded in $\Omega_s$.

 Put $\phi_-(u) = h(u) + \pi + \frac{1}{u^2 + 1}, \phi_+(u) = h(u) + 2\pi - \frac{1}{u^2+1}$. Note that  $\theta(u) = \pi - \frac{2}{u^2+1}$ and $\psi(u)   = h(u) + \frac{3\pi}{2}$.
 
 Let $x_0 = x_1 \le x_2$ and consider estimate \eqref{conf upper}

$$x_2 - x_1 \le \pi \intt_{u_1}^{u_2} \frac{du}{\theta(u)} + \intt_{u_1}^{u_2} \frac{\psi'(u)^2}{\theta(u)}du + \frac{\pi}{12}\intt_{u_1}^{u_2}\frac{\theta'(u)^2}{\theta(u)}du + O(1) = I_1 + I_2 + I_3 + O(1).$$

Direct computation shows that $I_1 + I_3 = u_2 - u_1 + O(1)$. Since $\int h'(u)^2du < \infty$ and $\theta(u) \ge \pi - 2$, we have $I_2 = O(1)$. Therefore we get
\begin{equation}\label{real bound}
x_2 \le u_2 - u_1 + x_1 + O(1) = u_2 + O(1).
\end{equation}

Put $e(z) = \exp(\exp(z)), z\in \bar{S_0}$. Note that $e$ is bounded on $\partial S_0$ and unbounded on $S_0$. Consider the function $b(w) = e(Z(w)), w\in \bar{S}$. This function is bounded on $\partial S$ and from the estimate \eqref{real bound} we get ${|b(u + iv)| \le \exp(C\exp(u))}$.

Put $\Omega = \{ re^{i\alpha}\mid \phi_-(\log r) < \alpha < \phi_+(\log r), r > 10\}$ and consider the function $g(z) = \frac{b(\log z)}{z^2}, z\in \bar \Omega$. We have $|g(z)| = O\left(\frac{1}{|z|^2}\right), z\in \partial \Omega$ and $|g(z)| \le \exp(C|z|), z\in \Omega$.

For  $z\in \Cm \backslash \bar \Omega$ put 

$$f(z) = \frac{1}{2\pi i}\intt_{\partial \Omega} \frac{g(w)}{z-w}dw.$$

\sloppy Since $|\phi_{\pm}'|$ are bounded, the  intersection of $\partial \Omega$ with the annulus ${\{ n < |z| < n+1\}}$ has length $O(1)$. Therefore this integral converges absolutely.

Note that by shifting contour  from $\partial \Omega$ to $(\partial \Omega \cap \Cm \backslash R\D) \cup (\partial R\D \cap \Omega)$ we can analytically extend $f$ to the disk $R\D$. Since $R$ can be chosen arbitrary large the function $f$ has an analytic extension to the whole plane. Moreover, choosing $R = 2|z|$ we obtain $|f(z)| \le 10\max\limits_{|w| = R} |g(w)| + O(1)$. Since ${|g(w)| \le \exp(C|w|)}$, the function $f$ is of finite exponential type. It remains to prove that $f$ is nonconstant and bounded in $\Omega_s$.

First of all we will prove that $\dist(\Omega, \Omega_s) = \eps > 0$. Note that $$\Omega = \left\{ re^{i\alpha} \mid s(r) + \pi + \frac{1}{\log^2 r + 1} < \alpha < s(r) + 2\pi - \frac{1}{\log^2 r + 1}, r > 10\right\},$$ therefore $\Omega \cap \Omega_s = \varnothing$.

Suppose that $z\in \Omega, w\in \Omega_s, |z - w| < \delta$. Then $||z| - |w|| < \delta$. Hence, 

 $$|\arg(z) - \arg(w)| \ge \frac{1}{\log^2 |z| + 1} - \delta\max_{|z| - \delta < r< |z| + \delta}s'(r).$$ From the boundedness of $h'(u)$ it follows that $s'(r) = O\left(\frac{1}{r}\right)$. Hence for small enough $\delta$ the arguments of $z$ and $w$ differ by at least $\frac{1}{2(\log^2 |z| + 1)}$. Thus $|z - w| \ge \frac{|z|}{20(\log^2 |z| + 1)} > \delta$ for small enough $\delta$.  We arrive at a contradiction.

Therefore for $z\in \Omega_s$, $|f(z)| \le \frac{1}{2\pi \eps} \intt_{\partial \Omega} |g(w)||dw| = C < \infty$.

To prove that $f$ is nonconstant we first note the following identity for $z\in \Omega$:
\begin{equation}\label{int plus}
f(z) = g(z) + \frac{1}{2\pi i}\intt_{\partial \Omega} \frac{g(w)}{z-w}dw.
\end{equation}

{Indeed, when the contour passes through the  point $z$ the residue at $z$ adds to the integral which can be easily seen to be equal to $g(z)$.}

Consider the line $\Lambda_0 = \R \subset S_0$. From Theorem \ref{arg asymp} under the mapping  $z\mapsto \exp(W(z))$ it is mapped onto the curve 

$$\Gamma = \{ re^{i\beta}, \beta = t(r) = s(r) + \frac{3\pi}{2} + o(1)\}.$$ 

{ Reasoning as before we can see that $dist(\Gamma, \partial \Omega) > 0$. Therefore the  integral in the \eqref{int plus} is $O(1)$.}

From Theorem \ref{Alf bound}  and definition of $e(z)$ we can see that ${|g(z)| \ge \frac{\exp(c|z|)}{|z|^2}}, z\in \Gamma$, for some $c > 0$. So $g(z)$ is unbounded on $\Gamma$. Therefore $f(z)$ is unbounded on $\Gamma$ as well and thus is nonconstant.
\end{proof}

\begin{remark}
{\normalfont It seems plausible that Theorem \ref{existence theorem} can be proved assuming only boundedness of $h'(x)$ instead of $h'(x)\to 0$, by modifying methods from \cite{War}. The only step where we used the more restrictive hypothesis is when we proved that $f$ is nonconstant and there we had a very big degree of freedom.}
\end{remark}
We will now proof a slightly weaker form of Theorem \ref{regular upper bound}.

\begin{theorem}\label{regular upper bound 2}
Let $\Omega_s$ be a rotating half-plane with rotation function $s$ and let $f$ be an entire function of finite exponential type bounded on $\Omega_s$. Assume that $s$ does not satisfy \eqref{sqint}, $h'(x) \to 0$ as $x\to \infty$ and 

$$\intt_0^\infty |h''(x)| dx < \infty.$$
Then $f$ is constant.
\end{theorem}

\begin{proof}
Let $S = \{ u + iv\mid h(u) + \pi < v < h(u) + 2\pi\}$, that is $\theta(u) = \pi$, $\psi(u) = h(u) + \frac{3\pi}{2}$.

The function $g(z) = f(\exp(W(z)))$ is analytic on $S_0$ and bounded on $\partial S_0$. Fix some big $u_1$ and consider estimate \eqref{conf lower}, 

$$x_2 - x_1 \ge u_2 - u_1 + \intt_{u_1}^{u_2} h'^2(u)du + O(1) = u_2 + \intt_0^{u_2}h'^2(u)du + O(1).$$

Since the function $f$ is of finite exponential type, we have

$$|g(x + iy)| \le C_1\exp(C_2\exp(U(z))) \le C_1\exp(C_2\exp(x_2 - \intt_0^{U(z)}h'(u)^2du + C_3)).$$

\sloppy Note that since the map $W$ is conformal and sends infinity to infinity, we have  $U(x+iy)\to +\infty$ as $x\to +\infty$. Therefore  $|g(x+iy)| = o(\exp(c\exp(x)))$ for every $c > 0$ as $x\to +\infty$. Thus, the function $b(z) = g(\log z), {\Re z \ge 0}$, is bounded on $\{ \Re z = 0\}$ and $b(z) = o(\exp(c|z|)), |z|\to \infty$. From the  Phragm{\'e}n-Lindel{\"o}f Theorem (see, e.g., \cite[Lecture 6, Theorem 3]{Levin}) it follows that $b$ is bounded. Therefore $f$ is bounded on $\Cm \backslash \bar\Omega_s$. Since by assumption $f$ is bounded on $\Omega_s$  we get that $f$ is constant by Liouville's Theorem.
\end{proof}

Now we will derive Theorem \ref{regular upper bound} from Theorem \ref{regular upper bound 2} and Theorem \ref{upper boundintr} (the proof of which we postpone to the next section).

\begin{proof}[Proof of Theorem \ref{regular upper bound}]

Since $h''(u) \in L^1(\R)$, there exists $\lim_{u\to \infty} h'(u) = H$.
 If $H = 0$, then we may apply Theorem \ref{regular upper bound 2} to conclude that $f$ is constant.

If $H \ne 0$ then we have $h(u) = (H + o(1))u \ne o(\sqrt{u})$. Therefore we can apply Theorem \ref{upper boundintr} and again conclude that $f$ is constant.
\end{proof}}
\subsection{Proof of Theorem \ref{upper boundintr}}
We will need the following estimate for harmonic measure in bounded simply connected domains:

\begin{theorem}[{\cite[Theorem 5.3]{harmonic measure}}]\label{measure estimate} Let $G$ be a simply connected domain bounded by a rectifiable Jordan curve and let $E\subset \partial G$ be an arc. If $\sigma$ is any curve connecting point $z_0\in G$ and $\partial G \backslash E$ which lies entirely inside of $G$, then
\begin{equation}\label{estimate}
\omega(z_0, E, G) \le \frac{8}{\pi}e^{-\pi \lambda(\sigma, E, G)},
\end{equation}
where $\omega(z_0, E, G)$ is the harmonic measure at $z_0$ of $E$ in $G$ and  ${\lambda(\sigma, E, G) = \frac{dist(\sigma, E)^2}{A(G)}}$, where $A(G)$ is the area of $G$ and $dist(\sigma, E)$ is the infimum of the lengths of paths which connect points on $\sigma$ with $E$ and lie entirely inside of $G$.
\end{theorem}

	{ In Theorem 5.3 from \cite{harmonic measure} we can vary not only $\sigma$ but also the metric $\rho$. Theorem \ref{measure estimate} is a special case with $\rho(z) = 1$.}
	
Put $h(x) = s(\exp(x)) + \pi$, ${V = \{ a + bi \mid a > 0, h(a) < b < h(a) + \pi \}}$. The function $g(z) = f(e^z)$ is defined on $\bar{V}$, bounded on $\partial V$ and  satisfies the estimate $|g(a + bi)| \le \exp(A + C\exp(a))$ for some ${A, C > 0}$.
Without loss of generality we may assume that ${|g(z)| \le 1}$ on $\partial V$. 
{Put ${G_t = V \cap \{ \Re z < t\}}$, ${E_t = [t + h(t)i, t + (h(t) + \pi)i]}$ for some ${t\in \R}$ and consider function ${u(z) = C\exp(t)\omega(z, E_t, G_t)}$. Obviously, we have ${A + u(z) \ge \log |g(z)|,\ z\in \partial G_t}$. Since $\log |g(z)|$ is a subharmonic function, this estimate holds for $z\in G_t$ as well. Thus, if for all $z_0 = a_0 + b_0i\in V$  we can prove that $u(z_0) \to 0,\ t\to \infty$, then we would have $|g(z_0)| \le \exp(A)$, from which it follows that $f$ is bounded on $\Cm\backslash \Omega_s$. Since $f$ is also bounded on $\Omega_s$, by Liouville's Theorem $f$ is constant. So it remains to prove that $u(z_0) \to 0,\ t\to \infty$.

Put $\sigma = [a_0 + h(a_0)i, a + bi]$. By the assumption of Theorem \ref{upper boundintr}, there exists a sequence of numbers $t_k$ tending to infinity such that $h(t_k) \ge c\sqrt{t_k}$ for some $c > 0$. Moreover, we may assume that $\frac{h(t_{k + 1})}{2} > h(t_k) + \pi + \frac{c^2}{8}$ and that $t_1 > a_0$. We have
\begin{multline*}
dist(\sigma, E) \ge (t - t_{n+1})+\summ_{k = 1}^n \sqrt{ (t_{k + 1} -t_k)^2 + (h(t_{k + 1}) - h(t_k) - \pi)^2} \ge \\ 
(t - t_{n+1})+\sum_{k = 1}^n \sqrt{ (t_{k + 1} - t_k)^2 + \frac{c^2}{4}t_{k + 1} + \frac{c^4}{64}} \ge  \\ (t - t_{n+1})+\sum_{k = 1}^n \left(t_{k + 1} - t_k + \frac{c^2}{8}\right) = t - t_1 + \frac{c^2n}{8},
\end{multline*}

where $n + 1$ is the number of $t_k$'s not greater than $t$. Using \eqref{estimate} we get

\begin{multline}\label{bound}
u(z_0) \le \frac{8C}{\pi}\exp \left(t - \pi \frac{(t - t_1 + \frac{c^2n}{8})^2}{t\pi}\right) \le \\ \frac{8C}{\pi}\exp\left(2\pi\left(t_1 - \frac{c^2n}{8}\right)\right).
\end{multline}
Obviously, when $n$ tends to infinity the right-hand side of \eqref{bound} tends to $0$, and since $n\to \infty$ as $t\to \infty$  we have that $u(z_0)\to 0$ as $t\to \infty$.
}
\qed

\end{document}